\documentclass[12pt,a4paper]{article}
\usepackage{a4wide}
\usepackage{amsmath,amsthm,amssymb}
\usepackage{amsmath,amscd}
\usepackage[all]{xy}
\usepackage{graphicx}

\newcommand{\define}[1]{\textbf{#1}}

\newtheorem{theorem}{THEOREM}[section]
\newtheorem{lemma}[theorem]{LEMMA}
\newtheorem{corollary}[theorem]{COROLLARY}
\newtheorem{property}[theorem]{PROPERTY}
\theoremstyle{definition}

\newtheorem{remark}[theorem]{Remark}

\DeclareMathOperator{\Sym}{Sym}
\DeclareMathOperator{\chara}{char}

\DeclareFontFamily{T1}{pzc}{}
\DeclareFontShape{T1}{pzc}{m}{it}{2 <-> pzcmi8t}{}
\DeclareMathAlphabet{\mathpzc}{T1}{pzc}{m}{it}
\newenvironment{frcseries}{\fontfamily{sffamily}\selectfont}{}
\newcommand{\textfrc}[1]{\underline{\frcseries#1}}
\newcommand{\mathfrc}[1]{\text{\textfrc{#1}}}

\title{A categorical approach to loops, neardomains and nearfields} 
\author{Philippe Cara\footnote{The first author was partially
    supported by grant 15.263.08 of the ``Fonds voor Wetenschappelijk
    Onderzoek-Vlaanderen''.}, Rudger Kieboom and  Tina Vervloet\\ 
Department of Mathematics\\
Vrije Universiteit Brussel\\
Pleinlaan 2\\
B-1050 Brussel\\
BELGIUM}

\begin{document}
\maketitle
\begin{abstract}
In this paper we study loops, neardomains and nearfields from a
categorical point of view. By choosing the right kind of morphisms, we
can show that the category of neardomains is equivalent to the
category of sharply 2-transitive groups. The other categories are also
shown to be equivalent with categories whose objects are sets of
permutations with suitable extra properties.

Up to now the equivalence between neardomains and sharply 2-transitive
groups was only known when both categories were equipped with the
obvious isomorphisms as morphisms. We thank Hubert Kiechle for this
observation~\cite{Kiec:pc}. 
\end{abstract}

\nocite{Kiec02,Kerb74,Verv09}

\section{Introduction}
\label{sec:intro}

Loops and nearfields are structures in algebra which generalize
groups and fields, respectively. The first examples of finite
proper nearfields were constructed by L.E. Dickson in~1905. Thirty
years later the finite nearfields were completely classified by
H. Zassenhaus. In~1965 H. Karzel introduced neardomains (which are a
weakening of nearfields) in such a way that there is a one-to-one correspondence with sharply $2$-transitive groups (see~\cite{Karz68}). At that time morphisms were not considered. The still unsolved problem is whether there exist proper neardomains, i.e. neardomains which are 
not nearfields. 

The link between loops and regular (i.e. sharply $1$-transitive)
permutation sets is of a similar but simpler kind.

The present paper describes these links in a uniform way. By
considering the right kind of morphisms we show that loops, 
resp. neardomains are categories equivalent to the categories of
regular permutation sets resp. sharply $2$-transitive groups. The
latter equivalence nicely restricts to an equivalence between the
category of nearfields and the category of sharply $2$-transitive
groups with the property that the translations form a subgroup.

In~\cite{Karz68} the correspondence between sharply $2$-transitive groups and 
neardomains is described. We define morphisms which turn this correspondence 
into an equivalence of categories. Kiechle~\cite{Kiec:pc} was also aware of an 
equivalence of categories but with a more restricted class of morphisms, 
namely isomorphisms.

\section{Loops and regular permutation sets}

\subsection{Loops}
A \define{loop} is a set $L$, together with a binary operation
$(k,l)\mapsto kl$ with identity satisfying the \emph{left and right loop
property}. This means that for every $a,b\in L$ there exist unique
elements $x,y\in L$ such that $ax=b$ and $ya=b$.

A \define{morphism of loops} $(L,.)\rightarrow (L',*)$ is a map
$f\colon L\rightarrow L'$ preserving the operations. This means that  
$f(a.b)=f(a)*f(b)$ for all $a,b\in L$. It follows that $f$ maps the
identity of $L$ onto the identity of $L'$.

We denote by $\mathfrc{Loop}$ the category of all loops together with
all morphisms of loops.
\subsection{Regular permutation sets}
Let $\Omega$ be a set and let $\Sym (\Omega)$ denote the set of all
permutations of $\Omega$. A \define{regular permutation set (r.p.s.)}
on $\Omega$ is
a subset $M$ of $\Sym(\Omega)$ such that the identity permutation 
$1_{\Omega}$ is in $M$ and 
$M$ acts regularly on $\Omega$, i.e. $\forall
  \alpha,\beta\in\Omega\colon \exists!\; m\in M\colon m(\alpha)=\beta$.

We construct a category $\mathfrc{rps}$ whose objects are triples $(M,\Omega,
\omega)$, where $M$ is an r.p.s. on $\Omega$ and $\omega\in\Omega$ is
a \define{base point}.

A \define{morphism of r.p.s.} $(M,\Omega, \omega)\rightarrow
(N,\Sigma,\sigma)$ is a pair $(f,\Phi)$ such that 
$f\colon M\rightarrow N$ and $\Phi\colon\Omega\rightarrow\Sigma$
  are maps satisfying $\Phi(\omega)=\sigma$ and 
$\forall m\in M, \forall\alpha\in\Omega\colon
  \Phi(m(\alpha))=(f(m))(\Phi(\alpha))$. The latter property can be
  summarized by the following commutative diagram in which the
  horizontal maps are the (left) actions of $M$ (resp. $N$) on
  the set $\Omega$ (resp. $\Sigma$).

\begin{displaymath}
\xymatrix{
M\ar[d]_{f\times\Phi}\times\Omega\ar[r] & {\Omega}\ar[d]^{\Phi}\\
N\times\Sigma\ar[r]& {\Sigma}}
\end{displaymath}

Composition of morphisms $(f,\Phi)\colon (M,\Omega,\omega)\rightarrow
(N,\Sigma,\sigma)$ and $(g,\Psi)\colon (N,\Sigma,\sigma)\rightarrow
(P,\Gamma,\tau)$ is defined by $(g\circ f,\Psi\circ\Phi)$. The
identity $1_{(M,\Omega,\omega)}$ is defined as $(1_M,1_{\Omega})$. 

One easily verifies that $\mathfrc{rps}$ is a category.

Notice that, by regularity, a morphism $(f,\Phi)\colon (M,\Omega,
\omega)\rightarrow (N,\Sigma,\sigma)$ satisfies $f(1_{\Omega})=1_{\Sigma}$. 

\subsection{Equivalence of the categories $\mathfrc{Loop}$ and $\mathfrc{rps}$}
The one-to-one correspondence between loops and regular permutation sets is folklore (see for instance~\cite{Capo03}). However we think it is useful to describe an explicit categorical equivalence.
\smallskip

Let $(M,\Omega,\omega)$ be an object of $\mathfrc{rps}$. By
regularity, the map $\mu\colon M\rightarrow\Omega\colon m\mapsto
m(\omega)$ is a bijection such that $\mu(1_{\Omega})=\omega$. 
Now define an operation 
$$
\begin{array}{rrcl}
\otimes_{\omega}\colon & M\times M & \rightarrow & M\\
& (m,n) & \mapsto & m\otimes_{\omega}n:=\mu^{-1}((m\circ n)(\omega))
\end{array}
$$

Notice that $m\otimes_{\omega}n$ is equivalently defined by
$(m\otimes_{\omega}n)(\omega)=(m\circ n)(\omega)$
(compare~\cite[p. 618]{Capo03}). 

It is easy to check that
\begin{property}
The pair $(M,\otimes_{\omega})$ is a loop with identity $1_{\Omega}$.
\end{property}

We can also construct a loop structure on the set $\Omega$. 
Let $(M,\Omega,\omega)$ be an object of $\mathfrc{rps}$. We still have
the bijection $\mu\colon M\rightarrow\Omega$. We define the operation
$$
\begin{array}{rrcl}
\cdot_{\omega}\colon &\Omega\times \Omega & \rightarrow & \Omega\\
&(\alpha,\beta) & \mapsto & \alpha\cdot_{\omega}\beta :=
(\mu^{-1}(\alpha)\otimes_{\omega}\mu^{-1}(\beta))(\omega) = (\mu^{-1}(\alpha)\circ\mu^{-1}(\beta))(\omega)
\end{array}
$$

One easily verifies the following
\begin{property}\label{prop0}
The pair $(\Omega,\cdot_{\omega})$ is a loop with identity $\omega$, and
$\mu\colon (M,\otimes_{\omega})\rightarrow(\Omega,\cdot_{\omega})$ is a
loop isomorphism.
\end{property}

The following property gives a useful characterization of morphisms of
regular permutation sets.

\begin{property}\label{prop1}
Let $(M,\Omega,\omega)$, $(N,\Sigma, \sigma)$ be objects of $\mathfrc{rps}$,
$f\colon M\rightarrow N$ and $\Phi\colon\Omega\rightarrow\Sigma$ maps
with $\Phi(\omega)=\sigma$. Then $(f,\Phi)$ is a morphism $(M,\Omega
,\omega)\rightarrow (N,\Sigma ,\sigma)$ if and only if the following conditions
are both satisfied.
\begin{enumerate}
\item\label{prop1a} $\forall m\in M\colon f(m)(\sigma)=\Phi(m(\omega))$
\item\label{prop1b} $f\colon (M,\otimes_{\omega})\rightarrow
  (N,\otimes_{\sigma})$ is a morphism of loops.
\end{enumerate}
\end{property}
\begin{proof}
Let  $(f,\Phi)\colon (M,\Omega,\omega)\rightarrow (N,\Sigma,\sigma)$
be a morphism, then (\ref{prop1a}) follows immediately, since $\forall
m\in M\colon f(m)(\sigma)=f(m)(\Phi(\omega))=\Phi(m(\omega))$. 

In order to prove (\ref{prop1b}) it suffices, by the regular action of
$N$ on $\Sigma$, to establish that $\forall m_1,m_2\in M\colon
(f(m_1\otimes_{\omega}
m_2))(\sigma)=(f(m_1)\otimes_{\sigma}f(m_2))(\sigma)$. \\ The left hand
side equals $\Phi((m_1\otimes_{\omega} m_2)(\omega)) = \Phi((m_1\circ
m_2)(\omega))=\Phi(m_1(m_2(\omega))) =$ $(f(m_1))(\Phi(m_2(\omega))) =$
$(f(m_1))((f(m_2))(\Phi(\omega))) = (f(m_1)\circ f(m_2))(\sigma) =
(f(m_1)\otimes_{\sigma}f(m_2))(\sigma)$. 

Conversely (\ref{prop1a}) implies that $f\colon M\rightarrow N$,
$\Phi\colon \Omega\rightarrow\Sigma$ with $\Phi(\omega)=\sigma$
satisfy $f(m)(\Phi(\alpha))=\Phi(m(\alpha))$ for all $m\in M$, and
for $\alpha=\omega$. We have to show that this condition holds for all
$\alpha\in \Omega$. For any such $\alpha$ there exists a unique $m'\in
M$ such that $m'(\omega)=\alpha$ (by the regular action of $M$ on
$\Omega$). Then it follows that for all $m\in M$ and all
$\alpha\in\Omega$ we have
$f(m)(\Phi(\alpha)) = f(m)(\Phi(m'(\omega))) \stackrel{(\ref{prop1a})}{=}
f(m)(f(m')(\sigma)) = (f(m)\circ f(m'))(\sigma) =
(f(m)\otimes_{\sigma} f(m'))(\sigma) \stackrel{(\ref{prop1b})}{=}
(f(m\otimes_{\omega}m'))(\sigma) \stackrel{(\ref{prop1a})}{=}
\Phi((m\otimes_{\omega}m')(\omega)) = \Phi((m\circ m')(\omega)) =
\Phi(m(\alpha))$.
\end{proof}

\begin{corollary}
Let $(f,\Phi)\in
\mathfrc{rps}((M,\Omega,\omega),(N,\Sigma,\sigma))$. Then $\Phi\colon
(\Omega,\cdot_{\omega})\rightarrow (\Sigma,\cdot_{\sigma})$ is a loop
  homomorphism. 
\end{corollary}

The proof easily follows from the definition of the operations
$\cdot_{\omega}$ and $\cdot_{\sigma}$ and Properties~\ref{prop0} and
\ref{prop1}. 

Now let $(\Omega,\cdot)$ be a loop. For $\alpha\in\Omega$ we  define $\lambda_{\alpha}\colon \Omega\rightarrow\Omega\colon\gamma\mapsto\alpha\cdot\gamma$. 

\begin{property}\label{newprop}
For a loop $(\Omega,\cdot)$ with identity $\omega$ we write $L=\{\lambda_{\alpha}\mid \alpha\in \Omega\}$, the set of left translations of $\Omega$. The triple $(L,\Omega,\omega)$ is a r.p.s..
\end{property}
\begin{proof}
Since
$(\Omega,\cdot)$ is a (left) loop, $L$ is a subset of
$\mathrm{Sym}(\Omega)$. Moreover $L$ acts regularly on $\Omega$ by the
(right) loop property of $\Omega$. Also
$1_{\Omega}=\lambda_{\omega}\in L$. Hence $(L,\Omega ,\omega)$ is a r.p.s..
\end{proof}

We now have a correspondence between the objects of $\mathfrc{rps}$ and $\mathfrc{Loop}$ which can be extended to an equivalence of categories.

\begin{theorem}
$$
\begin{array}{rccc}
F\colon & \mathfrc{rps}&\longrightarrow & \mathfrc{Loop}\\[.3cm]
& (M,\Omega, \omega)& \longmapsto & (\Omega,\cdot_{\omega})\\
& | & & |\\
&(f,\Phi) & \mapsto & \Phi\\
& \downarrow & & \downarrow\\
& (N,\Sigma,\sigma) & \longmapsto & (\Sigma, \cdot_{\sigma})
\end{array}
$$
is an equivalence of categories.
\end{theorem}
\begin{proof}
The functorial properties of $F$ are easily checked. We use
\cite[prop. 3.4.3(4)]{Borc94} to show that $F$ is an equivalence of categories.

$F$ is faithful when for every pair of objects $(M,\Omega, \omega)$,
$(N,\Sigma,\sigma)$ 
of $\mathfrc{rps}$ the induced map\\ $F_{(M,\Omega, \omega),
  (N,\Sigma,\sigma)}\colon \mathfrc{rps}((M,\Omega,
\omega),(N,\Sigma,\sigma))\rightarrow
\mathfrc{Loop}((\Omega,\cdot_{\omega}),(\Sigma,\cdot_{\sigma}))$ is
injective. \\
Let $(f,\Phi), (f',\Phi')\in  \mathfrc{rps}((M,\Omega,
\omega),(N,\Sigma,\sigma))$ be such that $F(f,\Phi)=F(f',\Phi')$. Then
$\Phi=\Phi'$ follows immediately.

Hence, for all $(m,\alpha)\in M\times\Omega$, we have that
$(f(m))(\Phi(\alpha))=\Phi(m(\alpha)) =
\Phi'(m(\alpha))=(f'(m))(\Phi'(\alpha))$. In particular, for
$\alpha=\omega$, it follows that
$\nu(f(m))=(f(m))(\sigma)=(f(m))(\Phi(\omega))=(f'(m))(\Phi'(\omega))=(f'(m))(\sigma)=\nu(f'(m))$
(where $\nu\colon N\rightarrow \Sigma\colon n\mapsto n(\sigma)$ is the
bijection analogous to $\mu$). Hence $f(m)=f'(m)$ for all $m\in M$ and
thus $f=f'$.

$F$ is full when all induced maps $F_{(M,\Omega, \omega),
  (N,\Sigma,\sigma)}$ (as above) are surjective. Let
$\Phi\in\mathfrc{Loop}((\Omega,\cdot_{\omega}),(\Sigma,\cdot_{\sigma}))$,
then $\Phi(\omega)=\sigma$ (since $\omega$ and $\sigma$ are identities in the
respective loops). Define $f_{\Phi}\colon M\rightarrow N\colon
m\mapsto (\nu^{-1}\circ\Phi\circ\mu)(m)$ (with $\mu$ and $\nu$ as
above). One uses Property~\ref{prop1} and the regularity of the action
of $N$ on $\Sigma$ to show that $(f_{\Phi},\Phi)\in
\mathfrc{rps}((M,\Omega, \omega),(N,\Sigma,\sigma))$. Clearly
$F(f_{\Phi},\Phi)=\Phi$. 

We now take a loop $(\Omega,\cdot)$ with identity $\omega$ with its left translations as defined above. We then know (Prop.~\ref{newprop}) that $(L,\Omega,\omega)$ is a r.p.s..
For $\alpha,\beta\in \Omega$ we have that
$(\lambda_{\alpha}\otimes_{\omega}\lambda_{\beta})(\omega) =
(\lambda_{\alpha}\circ\lambda_{\beta})(\omega)=\alpha\cdot(\beta\cdot\omega)
= \alpha\cdot\beta = (\alpha\cdot\beta)\cdot \omega =
\lambda_{\alpha\cdot\beta}(\omega)$. Again by regularity, it follows
that $\lambda_{\alpha}\otimes_{\omega}\lambda_{\beta} =
\lambda_{\alpha\cdot\beta}$. Hence $F(L,\Omega,\omega) =
(\Omega,\cdot_{\omega})$ where $\alpha\cdot_{\omega}\beta =
(\mu^{-1}(\alpha)\circ \mu^{-1}(\beta))(\omega) =
(\lambda_{\alpha}\circ\lambda_{\beta})(\omega)=\alpha\cdot\beta$, which
shows that $(\Omega,\cdot)=F(L,\Omega,\omega)$. So $F$ turns out to be
even strictly surjective on objects.
\end{proof}

\section{Neardomains, nearfields and sharply 2-transitive groups}

\subsection{Neardomains and nearfields}
A triple $(F,+,.)$ is said to be a \define{neardomain} if
\begin{enumerate}
\item\label{cond1} $(F,+)$ is a loop with neutral element $0$;
\item\label{cond2} $\forall a,b\in F\colon a+b=0\Rightarrow b+a=0$;
\item $(F\setminus\{0\},.)$ is a group (with neutral element $1$);
\item $\forall a\in F\colon 0.a=0$;
\item\label{cond5} $\forall a,b,c\in F\colon a.(b+c)=a.b+a.c$;
\item\label{rule6} $\forall a,b\in F\colon \exists d_{a,b}\in F\setminus\{0\}\colon
  (\forall x\in F\colon a+(b+x) = (a+b)+d_{a,b}.x)$.
\end{enumerate}
Notice that $(F,+,.)$ is a nearfield if and only if all $d_{a,b}$ are
$1$. In that case $(F,+)$ is a group. Also notice that \ref{cond1} and
\ref{cond5} imply that $\forall a\in F\colon a.0=0$. 

\begin{theorem}[see \cite{Karz68}, \cite{Kerb74} or \cite{Kiec02}]
A finite neardomain is a nearfield.
\end{theorem}

It is still an open problem whether there exists a (necessarily
infinite) neardomain which is not a nearfield.

We define the category $\mathfrc{n-Dom}$ of neardomains where objects
are neardomains and  morphisms are maps preserving both operations.

\begin{property}
Neardomain morphisms are injective.
\end{property}
\begin{proof}
Let $f\colon (F,+,.)\longrightarrow (F',+',.')$ be a morphism of
neardomains. 
Suppose $f(x)=f(y)$ for some $x,y\in F$. Since $(F,+,.)$ is a neardomain we
have by conditions \ref{cond1}, \ref{cond2} a unique additive inverse
$-y$ of $y$. Then, 
$$
f(x+(-y))=f(x)+'f(-y)=f(y)+'f(-y)=f(y+(-y))=f(0)=0'.
$$
If $x+(-y)\neq 0$ this element has an inverse $z$ in
$(F\setminus\{0\},.)$. Hence we get
$$
1'=f(1)=f((x+(-y)).z)=f(x+(-y)).'f(z)=0'.'f(z)=0',
$$ a contradiction.
\end{proof}

\begin{theorem}\label{trm:3.3}
Let $(F,+,.)$ be a neardomain and let
$$
T_2(F) = \{\tau_{a,b}\colon F\rightarrow F\colon x \mapsto a+bx\mid
a\in F, b\in F\setminus\{0\}\}
$$
Then $(T_2(F),\circ)$ is a group whose action on $F$ is sharply
$2$-transitive, i.e. for any two ordered
pairs $(\alpha_1,\alpha_2)$, $(\beta_1,\beta_2)$ of points of $F$
with $\alpha_1\neq \alpha_2$, $\beta_1\neq\beta_2$ there exists a
unique element $\tau_{a,b}\in T_2(F)$ such that
$\tau_{a,b}(\alpha_1)=\beta_1$ and $\tau_{a,b}(\alpha_2)=\beta_2$.  
\end{theorem}
\begin{proof}
See \cite[(5.1)]{Karz68}, \cite[(6.1)]{Kerb74} or \cite[(7.8)]{Kiec02}.
\end{proof}

\subsection{Sharply 2-transitive groups and involutions}\label{sec:3.2}

Let $G$ be a sharply $2$-transitive permutation group on a set
$\Omega$ with $\#\Omega \geq 2$.

We denote by $J$ the set of involutions in $G$, i.e.
$$
J = \{g\in G\mid g^2=1_{\Omega}\neq g\}
$$

One can quickly see that $J$ is never empty. Indeed, as $\#\Omega \geq 2$, we 
can take $\alpha\neq\beta$ in $\Omega$ and find a (unique) $g\in G$ with 
$g(\alpha)=\beta$ and $g(\beta) = \alpha$. Such $g$ must have order 
$2$ (since $g^2$ fixes both $\alpha$ and $\beta$, implying $g^2=1_{\Omega}$, 
by sharp $2$-transitivity). 

\begin{property}\label{prop:chara}
$J$ satisfies exactly one of the following properties:
\begin{enumerate}
\item every $j\in J$ has a unique fixpoint;
\item all $j\in J$ are fixpoint-free.
\end{enumerate}
In the latter case we say that $G$ has characteristic $2$ and write
$\chara G=2$. In the first case we put $\chara G\neq 2$. 
\end{property}
\begin{proof}
For a detailed proof, we refer to \cite[p.12]{Kerb74}, where the case $\chara G=2$, resp. $\chara G\neq 2$ is referred to as $G$ of type $0$, resp. of type  $1$. (The type refers to the number of fixpoints of an involution in $G$.) The main idea behind the proof is that sharp $2$-transitivity implies that a nontrivial element of $G$ cannot have $2$ (or more) fixed points. Moreover all elements of $J$ are conjugate in $G$, hence they have the same number of fixpoints.
\end{proof}

The reason
for this notation and for using the word \emph{characteristic} will be 
clarified when we establish the correspondence between sharply 2-transitive 
groups and neardomains (see Property~\ref{prop:3.6}).

In the case $\chara G\neq 2$ we write $\nu$ for the unique involution
fixing an arbitrarily chosen base point $\omega_0\in\Omega$. 

The following subset of $G$ plays an important role. We define $A\subseteq G$
as follows 
\begin{equation}\label{adef}
A=\begin{cases}
J\circ\nu & \text{ if }\chara G\neq 2\\
J\cup\{1_{\Omega}\}& \text{ if }\chara G=2
\end{cases}
\end{equation}

\begin{property}\label{prop:3.5}
The triple $(A,\Omega,\omega_0)$ is a regular permutation set on
$\Omega$. 
\end{property}
\begin{proof}
See \cite[(5.3)]{Karz68} or \cite[(3.3)]{Kerb74}.
\end{proof}

The category of sharply $2$-transitive groups will be denoted by
$\mathfrc{s2t-Gp}$. Its \define{objects} are quadruples
$(G,\Omega,\omega_0,\omega_1)$ where $G$ is a permutation group which
operates sharply 2-transitively on the set $\Omega$ (with $\#\Omega\geq
2$), with two different base points
$\omega_0$ and $\omega_1$ of $\Omega$. Morphisms will be defined
after Property~\ref{prop:3.6}.

Let $(G,\Omega, \omega_0, \omega_1)$ be an object in
$\mathfrc{s2t-Gp}$. On $\Omega$ we define an \define{addition} and a
\define{multiplication} as follows.
For $\alpha,\beta\in\Omega$ we define $\alpha +_0\beta$ to be
$a(\beta)$, where $a\in A$ is unique such that $a(\omega_0)=\alpha$.

Since the stabilizer $G_{\omega_0}$ is regular on $\Omega\setminus\{\omega_0\}$ we
can also define $\alpha \cdot_1\beta$ as $g(\beta)$ where $g\in G_{\omega_0}$
is unique such that $g(\omega_1)=\alpha$. We also put $\alpha \cdot_1\beta
=0$ when $\alpha =\omega_0$ or $\beta =\omega_0$.

\begin{property}\label{prop:3.6}
\begin{enumerate}
\item The triple $F=(\Omega, +_0,\cdot_1)$ is a neardomain.
\item $\chara G = 2\Leftrightarrow\chara F = 2$, i.e. $1+1=0$ in $F$
  ($1$ denotes the multiplicative identity of the neardomain $F$).
\end{enumerate}
\end{property}
\begin{proof}
\begin{enumerate}
\item See \cite[(5.2)]{Karz68} or \cite[(6.2)]{Kerb74}.
\item See \cite[(7.10)]{Kiec02}.
\end{enumerate}
\end{proof}

We now define the \define{morphisms} in $\mathfrc{s2t-Gp}$ to be pairs of maps
$(f,\Phi)\colon (G,\Omega,\omega_0,\omega_1)\rightarrow (H,\Sigma,\sigma_0,
\sigma_1)$ where either both $G$ and $H$ have characteristic~$2$ or
both have characteristic different
from~$2$ and where $f\colon G\rightarrow
H$ is a group homomorphism, $\Phi\colon \Omega\rightarrow \Sigma$ an
\define{injective} 
  map with $\Phi(\omega_0)=\sigma_0$ and $\Phi(\omega_1)=\sigma_1$, and
  such that the following diagram commutes. 
\begin{equation}\label{diagram2}
\xymatrix{
G\ar[d]_{f\times\Phi}\times\Omega\ar[r]^-{\lambda}
&{\Omega}\ar[d]^{\Phi} \\
H\times\Sigma\ar[r]_-{\mu}
& {\Sigma}}
\end{equation}
In this diagram $\lambda$ and $\mu$ denote the evaluation maps defined
by $\lambda (g,\omega)= g(\omega)$ and $\mu (h,\sigma)= h(\sigma)$. 

In the case  either $\chara G=2$
and $\chara H\neq 2$ or $\chara G\neq 2$ and $\chara H=2$, we put  $\mathfrc{s2t-Gp}((G,\Omega,\omega_0,\omega_1),
(H,\Sigma,\sigma_0, \sigma_1)) = \emptyset$ (see the second remark just below).
\begin{remark}
\begin{enumerate}
\item The injectivity of $\Phi$ implies the injectivity of $f$ since
$1_\Sigma=f(g)$ for some $g\in G$ implies that $\forall\omega\in\Omega$ we have
$\Phi(g(\omega))=f(g)(\Phi(\omega))=\Phi(\omega)$. Hence
$\forall\omega\in\Omega\colon g(\omega)=\omega$, proving $g=1_\Omega$.

\item The reason for not allowing any morphisms in the case where $G$
  and $H$ have different characteristics is the following. If
  $(f,\Phi)$ is a morphism from
  $(G,\Omega,\omega_0,\omega_1)$ to $(H,\Sigma,\sigma_0, 
\sigma_1)$ with either $\chara G=2$
and $\chara H\neq 2$ or $\chara G\neq 2$ and $\chara H=2$, the map
$\Phi$ between the associated neardomains $(\Omega,+_0,\cdot_1)$ and
$(\Sigma, +_{0'},\cdot_{1'})$ cannot be a morphism of neardomains.
 Indeed,
the first case  $\chara G=2\neq\chara H$ implies (by
Property~\ref{prop:3.6}.2.) that the multiplicative identities $1$ of
$(\Omega,+_0,\cdot_1)$ and $1'$ of $(\Sigma, +_{0'},\cdot_{1'})$
satisfy $1+1=0$ and $1'+1'\neq 0'$, which conflicts with $\Phi\colon (\Omega,
+_0,\cdot_1)\rightarrow (\Sigma, +_{0'},\cdot_{1'})$ being a morphism
of neardomains, as can be seen by $0'\neq 1'+1'=\Phi(1)+'\Phi(1) =
\Phi(1+1) = \Phi(0) = 0'$. Similarly the second case $\chara G\neq
2=\chara H$ leads to the contradiction $0'= 1'+'1'=\Phi(1)+'\Phi(1) =
\Phi(1+1) \neq \Phi(0) = 0'$, where we used the injectivity of
$\Phi$. 

Our elimination of ``bad'' morphisms in $\mathfrc{s2t-Gp}$ enables us to
show (in Property~\ref{prop:3.9}) that the corresponding 
$\Phi$ is a morphism of neardomains.
\end{enumerate}
\end{remark}

\begin{lemma}
For a morphism $(f,\Phi)\colon (G,\Omega,\omega_0,\omega_1)\rightarrow
(H,\Sigma,\sigma_0,\sigma_1)$ the following hold.
\begin{enumerate}
\item $f(J)\subseteq K$ (where $K$ denotes the set of involutions in
  $H$);
\item $f(A)\subseteq B$ where $A\subseteq G$ is defined in
  (\ref{adef}) above Property~\ref{prop:3.5} and $B\subseteq H$ is
  defined by  
$$B=\begin{cases}
K\circ f(\nu) & \text{ if }\chara H\neq 2\\
K\cup\{1_{\Sigma}\}& \text{ if }\chara H=2
\end{cases}$$
($f(\nu)$ is the unique element of $K$ fixing the base point
$\sigma_0=\Phi(\omega_0)$)
\end{enumerate}
\end{lemma}
\begin{proof}
\begin{enumerate}
\item Let $g\in J$, then $f(g)$ satisfies
  $(f(g))^2=f(g^2)=f(1_{\Omega})=1_{\Sigma}\neq f(g)$, since
  $f(1_{\Omega})=f(g)$ would imply (using the injectivity of $f$) that
  $g=1_{\Omega}$, contradicting the fact that $g\in J$. Hence
  $f(J)\subseteq K$.
\item By the previous remark 2., the existence of the morphism
  $(f,\Phi)$ implies that either $\chara G = \chara H\neq 2$ or
  $\chara G=\chara H=2$. In the first case it follows that
  $f(A)=f(J\circ \nu)=f(J)\circ f(\nu)\subseteq K\circ f(\nu)=B$. In
  the second case we have $f(A)=f(J\cup\{1_{\Omega}\}) =
  f(J)\cup\{f(1_{\Omega})\}\subseteq K\cup\{1_{\Sigma}\}=B$.
\end{enumerate}
\end{proof}

\begin{property}\label{prop:3.9}
For a morphism $(f,\Phi)\colon (G,\Omega,\omega_0,\omega_1)\rightarrow
(H,\Sigma,\sigma_0,\sigma_1)$ the map 
$\Phi\colon (\Omega,+_0,\cdot_1)\rightarrow
  (\Sigma,+_{0'},\cdot_{1'})$ is a morphism of neardomains.
\end{property}
\begin{proof}
Consider the diagram~(\ref{diagram2}) above involving the morphism
$(f,\Phi)$ from $(G,\Omega,\omega_0,\omega_1)$ to
$(H,\Sigma,\sigma_0,\sigma_1)$. For $\alpha,\beta\in\Omega$, $\alpha
+_0\beta = a(\beta)\in\Omega$ where $a\in A\subseteq G$ is unique such
that $a(\omega_0)=\alpha$. Now $\Phi(\alpha +_0\beta) =
\Phi(a(\beta))=\Phi(\lambda(a, \beta))=(\mu\circ(f\times
\Phi))(a,\beta) = \mu(f(a),\Phi(\beta))=f(a)(\Phi(\beta))$. On the
other hand, $\Phi(\alpha)+_{0'}\Phi(\beta)=b(\Phi(\beta))$ where $b\in
B\subseteq H$ is unique such that $b(\sigma_0)=\Phi(\alpha)$. 

Since $f(A)\subseteq B$ we also have $f(a)\in B$ and moreover
$(f(a))(\sigma_0)=\mu(f(a),\Phi(\omega_0))=\Phi(\lambda(a,\omega_0))=\Phi(a(\omega_0))=\Phi(\alpha)$. Since
$B$ acts regularly on $\Sigma$ it follows that $f(a)=b$. This implies
that $\Phi(\alpha
+_0\beta)=f(a)(\Phi(\beta))=b(\Phi(\beta))=\Phi(\alpha)+_{0'}\Phi(\beta)$. 

We recall that, for $\alpha,\beta\in\Omega$, 
$$
\alpha\cdot_1\beta =\left\{
\begin{array}{ll}
g(\beta) & \text{when }\alpha\in\Omega\setminus\{\omega_0\}\text{ with
} g\in G_{\omega_0}\text{ unique such that } g(\omega_1)=\alpha\\
\omega_0 & \text{when } \alpha=\omega_0 \text{ or } \beta=\omega_0
\end{array}\right.
$$
When $\alpha=\omega_0$ we have immediately
$\Phi(\alpha)\cdot_{1'}\Phi(\beta)=\sigma_0\cdot_{1'}\Phi(\beta)=\sigma_0=\Phi(\omega_0)=\Phi(\omega_0\cdot_1\beta)$.

When $\alpha\neq\omega_0$, we have
$\Phi(\alpha\cdot_1\beta)=\Phi(g(\beta))=\Phi(\lambda(g,\beta)) =
(\mu\circ(f\times\Phi))(g,\beta)=\mu(f(g),\Phi(\beta))$ (with $g\in G_{\omega_0}$
unique such that $g(\omega_1)=\alpha$). On the other hand,
$\Phi(\alpha)\cdot_{1'}\Phi(\beta)$ $=h(\Phi(\beta))$ with $h\in
H_{\sigma_0}$ unique such that $h(\sigma_1)=\Phi(\alpha)$. Note that,
by the injectivity of $\Phi$ we have
$\Phi(\alpha)\neq\Phi(\omega_0)=\sigma_0$. We must have $h=f(g)$ since
$f(g)(\sigma_1)=\mu(f(g),\Phi(\omega_1))=(\mu\circ(f\times\Phi))(g,\omega_1)=(\Phi\circ\lambda)(g,\omega_1)=\Phi(g(\omega_1))=\Phi(\alpha)=h(\sigma_1)$
and $f(g),h\in H_{\sigma_0}$. This implies
$\Phi(\alpha\cdot_1\beta)=\Phi(g(\beta))=\Phi(\lambda(g,\beta))=(\mu\circ(f\times
\Phi))(g,\beta)=\mu(f(g),\Phi(\beta))=f(g)(\Phi(\beta))=h(\Phi(\beta))=\Phi(\alpha)\cdot_{1'}\Phi(\beta)$.
Hence $\Phi$ is a neardomain homomorphism.
\end{proof}

For any neardomain $(F,+,\cdot)$ we can construct the object
$(T_2(F),F,0,1)$ in $\mathfrc{s2t-Gp}$.
It is clear that the stabilizer
of $0$ in $T_2(F)$ consists of the elements
$\tau_{0,\delta}$ with
$\delta\in F\setminus\{0\}$.

\begin{lemma}\label{lem:Achara}
When $G=T_2(F)$ we have $A=\{\tau_{\gamma,1}\mid
\gamma\in F\}$.
\end{lemma}
\begin{proof}
See \cite[(6.5)]{Kerb74}.
\end{proof}

\section{Equivalence of the categories $\mathfrc{s2t-Gp}$ and
  $\mathfrc{n-Dom}$} 
\begin{theorem}\label{trm:3.8}
The functor
$$
\begin{array}{rccc}
K\colon & \mathfrc{s2t-Gp}&\longrightarrow & \mathfrc{n-Dom}\\[.3cm]
& (G,\Omega, \omega_0,\omega_1)& \longmapsto & (\Omega,+_0,\cdot_1)\\
& | & & |\\
&(f,\Phi) & \mapsto & \Phi\\
& \downarrow & & \downarrow\\
& (H,\Sigma,\sigma_0,\sigma_1) & \longmapsto & (\Sigma, +_0,\cdot_1)
\end{array}
$$
is an equivalence of categories. Here $f\colon G\rightarrow H$,
$\Phi\colon \Omega\rightarrow\Sigma$ with $\Phi(\omega_0)=\sigma_0$
and $\Phi(\omega_1)=\sigma_1$ are defined as in
section~\ref{sec:3.2}. The neardomain operations
$+_0\colon\Omega\times\Omega\rightarrow\Omega$ and
$\cdot_1\colon\Omega\setminus\{\omega_0\}\times\Omega\setminus\{\omega_0\}\rightarrow
\Omega\setminus\{\omega_0\}$ are defined as in the previous section.
\end{theorem}
\begin{proof}
The shortest proof should consist in proving that $K$ is full and
faithful and essentially surjective on objects (\cite[prop.3.4.3
(4)]{Borc94}). However we prefer to show that there exists a functor
$L\colon \mathfrc{n-Dom}\rightarrow \mathfrc{s2t-Gp}$, which is of interest in
its own right, and two natural isomorphisms $1_{\mathfrc{n-Dom}}\cong
K\circ L$ and $L\circ K\cong 1_{\mathfrc{s2t-Gp}}$ (\cite[prop. 3.4.3
(3)]{Borc94}). The proof of functoriality of $K$ is left as an
exercise for the reader. The definition of $L\colon
\mathfrc{n-Dom}\rightarrow \mathfrc{s2t-Gp}$ is as follows. A
neardomain $(F,+,.)$ is sent to $L(F)=(T_2(F),F,0,1)$
where $(T_2(F),\circ)$, as defined in Theorem~\ref{trm:3.3}, is a
sharply $2$-transitive group acting on $F$ and $0$, $1$
are the identities of the loop $(F,+)$ and the group $(F\setminus\{0\},.)$
respectively. A morphism $\Phi\colon (F,+,.)\rightarrow (F',+',.')$ of
neardomains is sent to $L(\Phi)=(f_{\Phi},\Phi)\colon
(T_2(F),F,0,1) \rightarrow
(T_2(F'),F',0',1')$ where $f_{\Phi}\colon
(T_2(F),\circ)\rightarrow (T_2(F'),\circ)$ maps $\tau_{a,b}$ to
$\tau_{\Phi(a),\Phi(b)}$. Since $\Phi(1)=1'$ it immediately follows,
using Lemma~\ref{lem:Achara}, that $f_{\Phi}(A) =
\{f_{\Phi}(\tau_{a,1})\mid a\in F\} = \{\tau_{\Phi(a),1'}\mid a\in F\}
 \subseteq \{\tau_{a',1'}\mid a'\in F'\} = A'$. Moreover $f_{\Phi}$
is a group homomorphism since, on the
one hand, for $a,k\in F$ and $b,l\in F\setminus\{0\}$
we have
$f_{\Phi}(\tau_{a,b}\circ\tau_{k,l})=f_{\Phi}(\tau_{a+bk,d_{a,bk}bl}) =
\tau_{\Phi(a+bk),\Phi(d_{a,bk}bl)}=\tau_{\Phi(a)+\Phi(b)\Phi(k),\Phi(d_{a,bk})\Phi(b)\Phi(l)}$
while, on the other hand, $f_{\Phi}(\tau_{a,b})\circ
f_{\Phi}(\tau_{k,l})=\tau_{\Phi(a),\Phi(b)}\circ\tau_{\Phi(k),\Phi(l)}=\tau_{\Phi(a)+\Phi(b)\Phi(k),d_{\Phi(a),\Phi(b)\Phi(k)}'\Phi(b)\Phi(l)}$. 
Finally the missing link $\Phi(d_{a,bk})=d_{\Phi(a),\Phi(b)\Phi(k)}'$
is provided by applying $\Phi$ to the identity
$a+(bk+x)=(a+bk)+d_{a,bk}x$, which follows from axiom~\ref{rule6} for $F$, and
expanding $\Phi(a)+(\Phi(b)\Phi(k)+\Phi(x))$, again using rule~\ref{rule6} for
$F'$. 

Then it follows that $L(\Phi)$ is a morphism in $\mathfrc{s2t-Gp}$ by
easily verifying that for all $a\in F$,
$b\in F\setminus\{0\}$, $\omega\in F$ we have
$\Phi(\tau_{a,b}(\omega))=(f_{\Phi}(\tau_{a,b}))(\Phi(\omega))$.

The functoriality of $L$, i.e. $L(1_F)=1_F$ for every neardomain $F$
and $L(\Phi'\circ\Phi)=L(\Phi')\circ L(\Phi)$ for every composable
pair $\Phi\colon F\rightarrow F'$, $\Phi'\colon F'\rightarrow F''$ of
morphisms of neardomains, is easily verified.

In order to prove $K\circ L=1_{\mathfrc{n-Dom}}$ it suffices to check that $+_0$,
$\cdot_1$ (resp. $+_0'$, $\cdot_1'$) coincide with $+$, $\cdot$
(resp. $+'$, $\cdot'$).

We first look at the multiplication. Take any $\alpha$, $\beta$
in $F$. By definition $\alpha\cdot_1\beta$ equals
$\tau_{0,\alpha}(\beta)$ since $T_2(F)_{0}=\{\tau_{0,\delta}\mid
\delta\in F\setminus\{0\}\}$ and 
$\tau_{0,\delta}(1) = \alpha\Leftrightarrow
\delta=\alpha$. Hence we get $\alpha\cdot_1\beta=\alpha\cdot\beta$.

For the addition we consider $\alpha,\beta\in F$. By
Lemma~\ref{lem:Achara}, the unique element in $A$ mapping $0$
to $\alpha$ must be equal to $\tau_{\alpha,1}$. Therefore we
get $\alpha +_0\beta = \tau_{\alpha,1}(\beta) = \alpha +\beta$.

Secondly, we have to find a natural isomorphism
$\gamma\colon L\circ K\Rightarrow 1_{\mathfrc{s2t-Gp}}$. We define each component
$\gamma_{(G,\Omega,\omega_0,\omega_1)}\colon (L\circ
K)(G,\Omega,\omega_0,\omega_1) =
L(\Omega,+_0,\cdot_1)=(T_2(\Omega),\Omega,\omega_0,\omega_1)\rightarrow
(G,\Omega,\omega_0,\omega_1)$ as follows.

For each $\tau_{\alpha,\beta}\in T_2(\Omega)$ there exists, by the
sharp $2$-transitivity of $G$ on $(\Omega,+_0,\cdot_1)$ a unique
element $g\in G$, denoted by $g_{\alpha,\beta}$ such that
$g_{\alpha,\beta}(\omega_0)=\alpha=\tau_{\alpha,\beta}(\omega_0)$ and
$g_{\alpha,\beta}(\omega_1) =
\alpha+_0\beta\cdot_1\omega_1=\tau_{\alpha,\beta}(\omega_1)$. This
correspondence defines a bijective map $k\colon T_2(\Omega)\rightarrow
G\colon \tau_{\alpha,\beta}\mapsto g_{\alpha,\beta}$ and one verifies
  easily that $k\colon (T_2(\Omega),\circ)\rightarrow (G,\circ)$ is a
  group homomorphism (using also the sharp $2$-transitivity of
  $T_2(\Omega)$ on $\Omega$). Finally we define 
$\gamma_{(G,\Omega,\omega_0,\omega_1)}\colon (T_2(\Omega),\Omega,\omega_0,\omega_1)\rightarrow
(G,\Omega,\omega_0,\omega_1)$ by
$\gamma_{(G,\Omega,\omega_0,\omega_1)}(\tau_{\alpha,\beta},\omega)=(k(\tau_{\alpha,\beta}),\omega)=(g_{\alpha,\beta},\omega)$,
which is an isomorphism in $\mathfrc{s2t-Gp}$. The naturality of
$\gamma$ amounts to the commutativity of each square 
\begin{displaymath}
\xymatrix{
(T_2(\Omega),\Omega,\omega_0,\omega_1)\ar[d]_{(f_{\Phi},\Phi)}\ar[rrr]^-{\gamma_{(G,\Omega,\omega_0,\omega_1)}} 
&&&(G,\Omega,\omega_0,\omega_1)\ar[d]^{(f,\Phi)}\\
(T_2(\Sigma),\Sigma,\sigma_0,\sigma_1)\ar[rrr]^-{\gamma_{(H,\Sigma,\sigma_0,\sigma_1)}} 
&&&(H,\Sigma,\sigma_0,\sigma_1)}
\end{displaymath}
with $f\colon G\rightarrow H$ and $\Phi\colon \Omega\rightarrow\Sigma$
as in~\ref{sec:3.2} and $f_{\Phi}\colon T_2(\Omega)\rightarrow
T_2(\Sigma)\colon
\tau_{\alpha,\beta}\mapsto\tau_{\Phi(\alpha),\Phi(\beta)}$.

On the one hand we have, for each $(\tau_{\alpha,\beta},\omega)\in
T_2(\Omega)\times \Omega$, 
$((f,\Phi)\circ
\gamma_{(G,\Omega,\omega_0,\omega_1)})(\tau_{\alpha,\beta},\omega) =
(f,\Phi)(k(\tau_{\alpha ,\beta}),\omega) =
(f,\Phi)(g_{\alpha,\beta},\omega) =
(f(g_{\alpha,\beta}),\Phi(\omega))\in H\times\Sigma$. 

On the other
hand, $(\gamma_{(H,\Sigma,\sigma_0,\sigma_1)}\circ
(f_{\Phi},\Phi))(\tau_{\alpha,\beta},\omega)=$ $\gamma_{(H,\Sigma,\sigma_0,\sigma_1)}(f_{\Phi}(\tau_{\alpha,\beta}),\Phi(\omega))
=$\\
$\gamma_{(H,\Sigma,\sigma_0,\sigma_1)}(\tau_{\Phi(\alpha),\Phi(\beta)},\Phi(\omega)) 
= (g_{\Phi(\alpha),\Phi(\beta)},\Phi(\omega))\in H\times \Sigma$. 

Finally $f(g_{\alpha,\beta})=g_{\Phi(\alpha),\Phi(\beta)}$ since
$(f(g_{\alpha,\beta}))(\sigma_0)=(f(g_{\alpha,\beta}))(\Phi(\omega_0))=\Phi(\alpha)=g_{\Phi(\alpha),\Phi(\beta)}(\sigma_0)$
and, similarly, $(f(g_{\alpha,\beta}))(\sigma_1)=(\mu\circ
(f\times\Phi))(g_{\alpha,\beta},\omega_1) =
\Phi(g_{\alpha,\beta}(\omega_1)) =
\Phi(\alpha+_0\beta\cdot_1\omega_1)=\Phi(\alpha+_0\beta) =
\Phi(\alpha)+_{0'}\Phi(\beta) =
\Phi(\alpha)+_{0'}\Phi(\beta)\cdot_{1'}\sigma_1 =
\tau_{\Phi(\alpha),\Phi(\beta)}(\sigma_1) =
g_{\Phi(\alpha),\Phi(\beta)}(\sigma_1)$. Hence, again by sharp
$2$-transitivity, $f(g_{\alpha,\beta})=g_{\Phi(\alpha),\Phi(\beta)}$.

For each $\tau_{\alpha,\beta}\in T_2(\Omega)$ there exists, by sharp
$2$-transitivity of $G$ on $\Omega$ a unique $g_{\alpha,\beta}\in G$
such that $g_{\alpha,\beta}(\omega_0)=\tau_{\alpha,\beta}(0)=\alpha$,
and
$g_{\alpha,\beta}(\omega_1)=\tau_{\alpha,\beta}(1)=\alpha+_0\beta$.

For each morphism $(f,\Phi)\colon
(G,\Omega,\omega_0,\omega_1)\rightarrow (H,\Sigma,\sigma_0,\sigma_1)$
in $\mathfrc{s2t-Gp}$ we have to check the naturality condition 
\begin{equation*}
(f,\Phi)\circ\gamma_{(G,\Omega,\omega_0,\omega_1)} =
\gamma_{(H,\Sigma,\sigma_0,\sigma_1)}\circ (f_{\Phi},\Phi)\tag{$*$}
\end{equation*}
where $f_{\Phi}\colon (T_2(\Omega),\circ)\rightarrow
(T_2(\Sigma),\circ)$ maps $\tau_{\alpha,\beta}\colon
\Omega\rightarrow\Omega$ to
$\tau_{\Phi(\alpha),\Phi(\beta)}\colon\Sigma\rightarrow\Sigma$.

The left hand side of $(*)$ sends $\tau_{\alpha,\beta}\in T_2(\Omega)$
to $(f,\Phi)(g_{\alpha,\beta})=f(g_{\alpha,\beta})\in H$. The right
hand side maps $\tau_{\alpha,\beta}$ to
$\gamma_{(H,\Sigma,\sigma_0,\sigma_1)}(f_{\Phi}(\tau_{\alpha,\beta}))
=
\gamma_{(H,\Sigma,\sigma_0,\sigma_1)}(\tau_{\Phi(\alpha),\Phi(\beta)})
= h_{\Phi(\alpha),\Phi(\beta)}\in H$.

Clearly $f(g_{\alpha,\beta})=h_{\Phi(\alpha),\Phi(\beta)}$ since both
elements of the sharply $2$-transitive group $H$ on $\Sigma$ send
$\sigma_0$ to $\Phi(\alpha)$ and $\sigma_1$ to $\Phi(\beta)$.

Hence $\gamma\colon L\circ K\Rightarrow 1_{\mathfrc{s2t-Gp}}$ is a
natural transformation.
\end{proof}

Since a full and faithful functor reflects isomorphisms we immediately get 

\begin{property}
Let $(G,F)$ and $(G',F')$ be sharply $2$-transitive permutation
groups. Then $(G,F)$ and $(G',F')$ are isomorphic as permutation
groups if and only if the associated neardomains $(F,+,.)$ and
$(F',+',.')$ are isomorphic in $\mathfrc{n-Dom}$.
\end{property}
(For a noncategorical proof, see \cite[(6.3)]{Kerb74}.)

The equivalence obtained in Theorem~\ref{trm:3.8} can be restricted to
interesting subcategories of $\mathfrc{s2t-Gp}$ and
$\mathfrc{n-Dom}$ respectively, which sheds new light on the possible
difference between neardomains and nearfields.

\begin{theorem}\label{trm3.10}
Let $\mathfrc{s2t-Gp}_{A}$ be the full subcategory of $\mathfrc{s2t-Gp}$
on objects $(G,\Omega,\omega_0,\omega_1)$ in which the subset $A$ (as
defined in~(\ref{adef})) is a \emph{subgroup} of $G$.

Then the functor $K$ restricts (and corestricts) to an equivalence 
$$
K_A\colon \mathfrc{s2t-Gp}_A\longrightarrow\mathfrc{n-Fld},
$$
where $\mathfrc{n-Fld}$ denotes the full subcategory of
$\mathfrc{n-Dom}$ with nearfields as objects. 
\end{theorem}
\begin{proof}
In Theorem (7.1) of~\cite{Karz68} it is shown that, for a sharply
$2$-transitive group on a set $\Omega$, the associated neardomain is a
nearfield if and only if $J^2=\{gh\mid g,h\in J\}$ (with $J$ the set
of involutions in $G$) is a subgroup of $G$. In connection with
Theorem (3.7) of~\cite{Karz68} it follows that $J^2$ is a subgroup of
$G$ if and only if $A$ is a subgroup of $G$.

Thus $K_A\colon \mathfrc{s2t-Gp}_A\longrightarrow\mathfrc{n-Fld}$ is a
functor into the category of nearfields.

On the other hand, the functor $L\colon \mathfrc{n-Dom}\rightarrow
\mathfrc{s2t-Gp}$ restricts (and corestricts) to a functor $L_A\colon
\mathfrc{n-Fld}\rightarrow\mathfrc{s2t-Gp}_A$. It suffices to notice
that when $(F,+,.)$ is a nearfield, $(F,+)$ and $(F\setminus\{0\},.)$
are groups. Hence $(A,\circ) = (\{\tau_{\gamma,1}\mid \gamma\in
F\},\circ)$ is a group since, for each $\alpha,\beta, x\in F$ we have
$(\tau_{\alpha,1}\circ\tau_{\beta,1})(x) =
\tau_{\alpha,1}(\beta+x)=\alpha+(\beta+x)=(\alpha+\beta)+x =
\tau_{\alpha+\beta,1}(x)$, which shows that
$\tau_{\alpha,1}\circ\tau_{\beta,1}=\tau_{\alpha+\beta,1}$. Similarly
$\tau_{\alpha,1}^{-1}=\tau_{-\alpha,1}$.

Finally one verifies that $K_A\circ L_A = 1_{\mathfrc{n-Fld}}$ (as in
the proof of Theorem~\ref{trm:3.8}) and that
$\gamma_A=(\gamma_{(G,\Omega,\omega_0,\omega_1)})_{(G,\Omega,\omega_0,\omega_1)}\colon
  L_A\circ K_A\Rightarrow 1_{\mathfrc{s2t-Gp}_A}$ is a natural
  isomorphism (as in the proof of Theorem~\ref{trm:3.8}).
\end{proof}
\textbf{Acknowledgement. }
We thank the referee for helpful suggestions.

\end{document}